\documentclass[11pt]{amsart}
\usepackage[margin=1in]{geometry}
\usepackage{amsfonts}
\usepackage{amsmath}
\usepackage{amssymb}
\usepackage{amsthm}
\usepackage{enumerate}
\makeatletter
\def\imod#1{\allowbreak\mkern10mu({\operator@font mod}\,\,#1)}
\makeatother

\usepackage{graphicx}
\usepackage{tabu}
\usepackage{slashbox}
\usepackage[hang,flushmargin]{footmisc} 

\usepackage{mathrsfs}

\newtheorem{theorem}{Theorem}[section]
\newtheorem{proposition}{Proposition}[section]
\newtheorem{lemma}{Lemma}[section]
\newtheorem{corollary}{Corollary}[section]
\newtheorem{fact}{Fact}[section]
\newtheorem{conjecture}{Conjecture}[section]

\theoremstyle{definition}
\newtheorem{definition}{Definition}[section]

\newtheorem{example}{Example}[section]

\begin{document}

\title{Flexible Toggles and Symmetric Invertible Asynchronous Elementary Cellular  Automata}
\author{Colin Defant}
\address{Old Address: University of Florida \\ 1400 Stadium Rd. \\ Gainesville, FL 32611 United States}
\address{Current Address: Princeton University \\ Fine Hall, 304 Washington Rd. \\ Princeton, NJ 08544}
\email{cdefant@princeton.edu}
\thanks{
This work was supported by National Science Foundation grant DMS--1358884.}

\begin{abstract}  
A sequential dynamical system (SDS) consists of a graph $G$ with vertices $v_1,\ldots,v_n$, a state set $A$, a
collection of ``vertex functions" $\{f_{v_i}\}_{i=1}^n$, and a permutation $\pi\in S_n$ that specifies how to compose these
functions to yield the SDS map $[G,\{f_{v_i}\}_{i=1}^n,\pi]\colon A^n\to A^n$. In this paper, we study symmetric invertible SDS defined over the cycle graph $C_n$ using the set of states $\mathbb F_2$. These are, in other words, asynchronous elementary cellular automata (ECA) defined using ECA rules 150 and 105. Each of these SDS defines a group action on the set $\mathbb F_2^n$ of $n$-bit binary vectors. Because the SDS maps are products of involutions, this relates to \emph{generalized toggle groups}, which Striker recently introduced. In this paper, we further generalize the notion of a generalized toggle group to that of a \emph{flexible toggle group}; the SDS maps we consider are examples of Coxeter elements of flexible toggle groups. 

Our main result is the complete classification of the dynamics of symmetric invertible SDS defined over cycle graphs using the set of states $\mathbb F_2$ and the identity update order $\pi=123\cdots n$.  More precisely, if $T$ denotes the SDS map of such an SDS, then we obtain an explicit formula for $\vert\text{Per}_r(T)\vert$, the number of periodic points of $T$ of period $r$, for every positive integer $r$. It turns out that if we fix $r$ and vary $n$ and $T$, then $\vert\text{Per}_r(T)\vert$ only takes at most three nonzero values. 
\end{abstract} 

\maketitle 

\bigskip



\section{Introduction}
Suppose we wish to model a finite system in which objects have various states and can update their states in discrete time steps. Moreover, assume that the state to which an object updates depends only on the current state of that object along with the states of other nearby or connected objects. This rough idea is the basis for a graph dynamical system. A graph dynamical system contains a base graph used to represent the connectivity of objects, a set of states, and a collection of vertex functions that determine how a vertex updates its state depending on its current state and the current states of its neighbors. In addition, a graph dynamical system contains some rule determining the scheme by which vertices update their states.  In a series of papers published between 1999 and 2001, Barrett, Mortveit, and Reidys introduced the notion of a sequential dynamical system (SDS), a graph dynamical system in which vertices update their states sequentially \cite{Barrett1,Barrett2,Barrett3,Mortveit2}. Subsequently, several researchers have worked to develop a general theory of SDS (see, for example, \cite{Barrett4,Barrett5,Barrett6,Collina,Garcia1,Macauley,
Macauley2,Reidys1}). The recent article \cite{Collina} is particularly interesting because it shows how SDS, originally proposed as models of computer simulation, are now being studied in relation with Hecke-Kiselman monoids in algebraic combinatorics. We draw most of our terminology and background information concerning SDS from \cite{Mortveit}, a valuable reference for anyone interested in exploring this field.

An SDS is built from the following parts:
\begin{itemize}
\item An undirected simple graph $G$ with vertices $v_1,\ldots,v_n$.
\item A set of states $A$. We typically let $A=\mathbb F_2$ (the field with $2$ elements). Let $q(v_i)$ denote the state of vertex $v_i$. 
\item A collection of \emph{vertex functions} $\{f_{v_i}\}_{i=1}^n$. Each vertex $v_i$ of $G$ is endowed with its own vertex function $f_{v_i}\colon A^{\deg(v_i)+1}\rightarrow A$. 
\item A permutation $\pi\in S_n$. The permutation $\pi$ is known as the \emph{update order} of the SDS. 
\end{itemize}

Our attention will be directed primarily toward SDS defined over the cycle graph $C_n$. We label the vertices of $C_n$ as $v_1,\ldots,v_n$, where $v_i$ is adjacent to $v_j$ if and only if $i-j\equiv\pm 1\pmod n$. 

The following notation and definitions will prove useful. We always assume $n\geq 3$ is an integer. Let $S_\Omega$ denote the symmetric group on a set $\Omega$, and let $S_m=S_{\{1,\ldots,m\}}$. If $p$ is a polynomial or a vertex of a graph, let $\deg(p)$ denote the degree of $p$. A function $g\colon\mathbb F_2^m\rightarrow\mathbb F_2$ is called \emph{symmetric} if $g(x_1,\ldots,x_m)=g(x_{\sigma_1},\ldots,x_{\sigma_m})$ for any $(x_1,\ldots,x_m)\in\mathbb F_2^m$ and any permutation $\sigma_1\cdots\sigma_m\in S_m$ (in this article, we write permutations as words). If $X$ is a set and $f\colon X\to X$ is a function, we say an element $x$ of $X$ is a periodic point of $f$ of period $r$ if $f^r(x)=x$ and $f^s(x)\neq x$ for all positive integers $s<r$. Define $\text{Per}_r(f)$ to be the set of periodic points of $f$ of period $r$. If $X$ is finite, then $\vert\text{Per}_1(f^r)\vert=\sum_{d\mid r}\vert\text{Per}_d(f)\vert$. By M\"obius inversion, 
\begin{equation}\label{Eq3}
\vert\text{Per}_r(f)\vert=\sum_{d\mid r}\mu(r/d)\vert\text{Per}_1(f^d)\vert,
\end{equation} where $\mu$ is the number-theoretic M\"obius function. 

Consider a general SDS as defined above. Suppose a vertex $v_i$ is adjacent to vertices $v_{j_1},\ldots,v_{j_d}$, where $d=\deg(v_i)$. When it is time for $v_i$ to update, the vertex function $f_{v_i}$ takes as its inputs the states $q(v_i),q(v_{j_1}),\ldots,q(v_{j_d})$ and outputs the new state to which $v_i$ should update. How should the states $q(v_i),q(v_{j_1}),\ldots,$ $q(v_{j_d})$ be ordered when used as arguments of the function $f_{v_i}$? In general, the answer to this question can vary depending on the problem one wishes to solve or the behavior one wishes to model. Here, we need not to worry about this issue because we will only deal with symmetric vertex functions.

The vector $(q(v_1),\ldots,q(v_n))$, which lists all of the states of the vertices of $G$ in the order corresponding to the vertex indices, is known as the \emph{system state} of the SDS. For each vertex $v_i$ of $G$, we define the local update function $F_{v_i}\colon\mathbb F_2^n\rightarrow\mathbb F_2^{n}$ by \[F_{v_i}(x_1,\ldots,x_n)=(x_1,\ldots,x_{i-1},f_{v_i}(x_i,x_{j_1},\ldots,x_{j_d}),x_{i+1},\ldots,x_n),\] where $v_{j_1},\ldots,v_{j_d}$ are the neighbors of $v_i$. The local update function $F_{v_i}$ represents the change in the system state that occurs when the vertex $v_i$ updates its state using the vertex function $f_{v_i}$. From the local update functions and the update order $\pi=\pi_1\cdots\pi_n$, we define the SDS map $F\colon\mathbb F_2^n\rightarrow\mathbb F_2^n$ by \[F=F_{v_{\pi_n}}\circ\cdots\circ F_{v_{\pi_1}}.\] Applying the map $F$ is known as a \emph{system update}; it represents the change in the system state that occurs when each vertex $v_i$ updates in the sequential order specified by the permutation $\pi$.  

To emphasize that an SDS is defined over the graph $G$ using the collection of vertex functions $\{f_{v_i}\}_{i=1}^n$ and the update order $\pi$, it is common to denote the SDS map by $[G,\{f_{v_i}\}_{i=1}^n,\pi]$. If all of the vertex functions $f_{v_i}$ are equal to the same function $f$, we simply write $[G,f,\pi]$ for the SDS map. 
\begin{example} \label{Example1}
Consider the graph $C_4$ shown in block A of Figure \ref{Fig1}. 
\begin{figure}
\begin{center} 
\includegraphics[width=\textwidth]{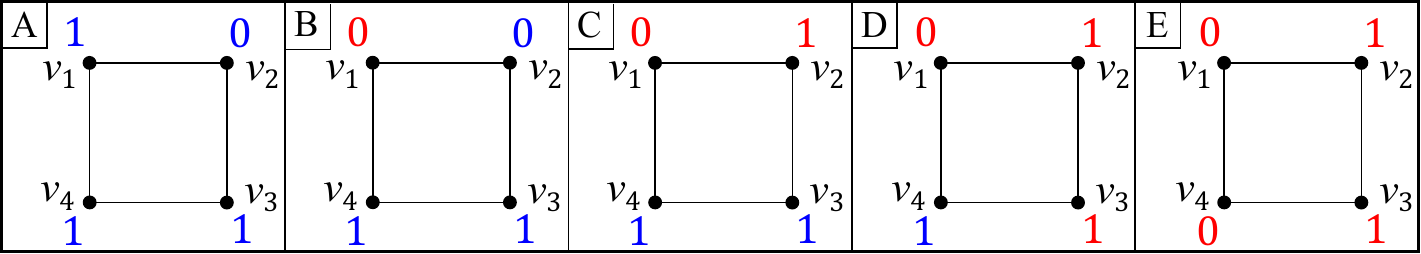}
\end{center}
\caption{\footnotesize{A system update of the SDS of Example \ref{Example1}. Block A shows the inital state of the SDS. Blocks B, C, and D show the intermediate steps of the system update. Block E shows the system state obtained after completing the system update.}} \label{Fig1} 
\end{figure} 
Define an SDS over $C_4$ using the update order $\pi=1234$, the set of states $\mathbb F_2$, and the vertex function $f$ given by $f(x,y,z)=x+y+z$. The initial system state of this SDS is $(1,0,1,1)$, as shown by the blue labels in block A of Figure \ref{Fig1}. We will perform a system update in order to find $[C_4,f,\pi](1,0,1,1)$. Because $\pi_1=1$, we first apply the local update function $F_{v_1}$. We have \[F_{v_1}(x_1,x_2,x_3,x_4)=(f(x_4,x_1,x_2),x_2,x_3,x_4),\] so \[F_{v_1}(1,0,1,1)=(f(1,1,0),0,1,1)=(0,0,1,1).\] Similarly, one can calculate that $F_{v_2}(0,0,1,1)=(0,1,1,1)$, $F_{v_3}(0,1,1,1)=(0,1,1,1)$, and \\$F_{v_4}(0,1,1,1)=(0,1,1,0)$. Letting $F=[C_4,f,\pi]$, we find that $F(1,0,1,1)=(0,1,1,0)$, as shown in block E of the figure. In other words, through a sequence of local updates, the system update transforms the system state $(1,0,1,1)$ into the new system state $(0,1,1,0)$. \hfill $\blacklozenge$
\end{example} 
The SDS map $[G,\{f_{v_i}\}_{i=1}^n,\pi]$ tells us how the states of the vertices of the graph $G$ change when we update the graph in a sequential manner. A natural aim is to analyze the long-term behavior of the states of the vertices as we continually update the system. To do this, we make use of a directed graph known as the \emph{phase space} of the SDS map. The phase space of a function $F\colon X\to X$, denoted $\Gamma(F)$, is a directed graph with vertex set $V(\Gamma(F))=X$ and edge set \[E(\Gamma(F))=\{(x,y)\in X\times X\colon y=F(x)\}.\] In other words, a directed edge is drawn from $x$ to $F(x)$ for each $x\in X$. The phase space of the SDS map given in Example \ref{Example1} is shown in Figure \ref{Fig2}. 
\begin{figure}[ht]
\begin{center} 
\includegraphics[height=47mm]{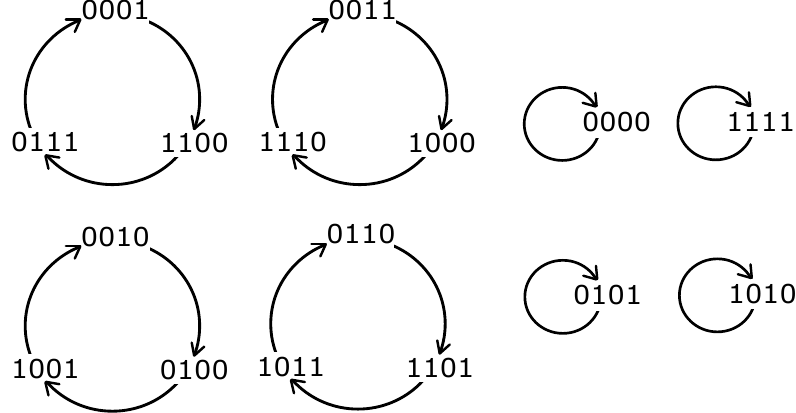}
\end{center}
\caption{\footnotesize{The phase space of the SDS map from Example \ref{Example1}. Note that we have written vectors as binary strings in order to improve the aesthetics of the image.}} 
\label{Fig2} 
\end{figure} 

In Example \ref{Example1}, we used the update order $\pi=1234$. Henceforth, we let $\text{id}$ denote the identity permutation $123\cdots n$; this will be the primary update order we will consider. Our example also used the vertex function $f\colon\mathbb F_2^3\rightarrow\mathbb F_2$ given by $f(x,y,z)=x+y+z$. Following \cite{Mortveit}, we denote this specific function by $\text{parity}_3$. Similarly, we define the function $(1+\text{parity})_3\colon\mathbb F_2^3\rightarrow\mathbb F_2$ by \[(1+\text{parity})_3(x,y,z)=1+\text{parity}_3(x,y,z)=1+x+y+z.\] More generally, for a positive integer $k$, the function $\text{parity}_k\colon\mathbb F_2^k\to\mathbb F_2$ is defined as the sum of the bits modulo $2$, and $(1+\text{parity})_k\colon\mathbb F_2^k\to\mathbb F_2$ is its negation.

We say an SDS is \emph{symmetric} if all of its vertex functions are symmetric (as mentioned before, this is the only type of SDS we will consider). Call an SDS \emph{invertible} if its SDS map is a bijection. An SDS defined over a cycle graph using the set of states $\mathbb F_2$ is also known as an \emph{asynchronous elementary cellular automaton}. The goal of this paper is to fully describe the dynamics of symmetric invertible asynchronous elementary cellular automata (SIAECA) with the identity update order. The following proposition, which provides one reason why the maps $\text{parity}_k$ and $(1+\text{parity})_k$ merit special attention in the theory of SDS, will help us accomplish our goal. 

\begin{proposition}[{\cite[Prop. 4.16]{Mortveit}}]\label{Thm1}
Suppose $\{f_{v_i}\}_{i=1}^n$ is a collection of symmetric vertex functions. The SDS map $[G,\{f_{v_i}\}_{i=1}^n,\pi]\colon\mathbb F_2^n\rightarrow\mathbb F_2^n$ is bijective if and only if \[f_{v_i}\in\{\text{\emph{parity}}_{\deg(v_i)+1},\text{\emph{(1+parity)}}_{\deg(v_i)+1}\}\] for each $i\in\{1,\ldots,n\}$.  
\end{proposition}

There are a total of $256$ functions from $\mathbb F_2^3$ to $\mathbb F_2$. These are the possible vertex functions that can be used in an elementary cellular automaton, so they are called the elementary cellular automaton (ECA) rules \cite{Mortveit}. These rules are typically numbered with the integers from $0$ to $255$; the functions $\text{parity}_3$ and $(1+\text{parity})_3$ are ECA rules $150$ and $105$, respectively. 

Two functions $f\colon X\to X$ and $g\colon Y\to Y$ are said to be \emph{dynamically equivalent} (or \emph{conjugate}) if there is a bijection $h\colon X\to Y$ such that $h\circ f=g\circ h$. If $f$ and $g$ are dynamically equivalent, then their phase spaces are isomorphic as directed graphs. If $X$ and $Y$ are finite, then the maps $f$ and $g$ are said to be \emph{cyclically equivalent} if $\vert\text{Per}_r(f)\vert=\vert\text{Per}_r(g)\vert$ for every positive integer $r$. The notions of dynamical equivalence and cyclical equivalence coincide if $X$ and $Y$ are finite and $f$ and $g$ are bijections. 

Let us fix a base graph $C_n$ and a collection $\{f_{v_i}\}_{i=1}^n$ of vertex functions that are all either $\text{parity}_3$ or $(1+\text{parity})_3$. In Section 4, we consider only the identity update order $\text{id}$. It may seem as though this restricts our results to only one of $n!$ different update orders. However, some theorems of Macauley and Mortveit \cite{Macauley2} show that varying the update order can give rise to at most $\left\lfloor n/2\right\rfloor$ dynamically nonequivalent SDS maps. Indeed, these authors show that $\delta(C_n)$, the number of connected components of a certain graph $D(C_n)$, is an upper bound for the number of dynamically nonequivalent SDS maps obtainable by varying the update order. Combining Propositions 5.1 and 5.9 from \cite{Macauley2} shows that $\delta(C_n)=\left\lfloor n/2\right\rfloor$. 

The remainder of the paper is organized as follows. In Section 2, we discuss the notions of ``toggles" and ``generalized toggles," which have attracted recent attention among combinatorialists. We further generalize the notion of a generalized toggle to that of a ``flexible toggle" and show that symmetric invertible SDS provide examples of flexible toggles. In Section 3, we recall some facts from linear algebra and discuss how they can be used to analyze the dynamics of affine self-maps of $\mathbb F_2$-vector spaces. In Section 4, we give a formula for $\vert\text{Per}_r(T)\vert$ for any positive integer $r$ when $T$ is the SDS map of an SIAECA with the identity update order. To prove our theorems, we borrow many techniques from Elspas \cite{Elspas} and Milligan and Wilson \cite{Milligan} (which we discuss in Section 3). In \cite{Elspas} and \cite{Milligan}, the authors determine the dynamics of affine maps from the elementary divisors of certain linear maps. Our methods are slightly different, and they are notable because they illustrate that it is not actually necessary to know all of the elementary divisors of these linear maps. With the help of M\"obius inversion, we need only know the invariant factors of these linear maps. In other words, it is not necessary to factor the invariant factors into irreducibles.    

\section{Toggles}
Suppose $E$ is a set and $\mathscr L$ is a fixed collection of subsets of $E$. For each $e\in E$, define the \emph{generalized toggle} $t_e\colon\mathscr L\to\mathscr L$ by \[t_e(X)=\begin{cases} X\cup\{e\}, & \mbox{if } e\not\in X\text{ and }X\cup\{e\}\in\mathscr L; \\ X\setminus\{e\}, & \mbox{if } e\in X\text{ and }X\setminus\{e\}\in\mathscr L; \\ X, & \mbox{otherwise}. \end{cases}\] The \emph{generalized toggle group} of $E$ with respect to $\mathscr L$ is the subgroup of the symmetric group $S_\mathscr{L}$ generated by $\{t_e\colon e\in E\}$. In the case that $E$ is a partially ordered set and $\mathscr L$ is the collection of order ideals of $E$, the functions $t_e$ are simply called \emph{toggles}. In this case, the generalized toggle group of $E$ with respect to $\mathscr L$ is known as the \emph{toggle group} of $E$. 

Toggles and toggle groups were originally studied by Cameron and Fon-der-Flaass in 1995 \cite{Cameron}, though they were not given these names until the article of Striker and Williams in 2012 \cite{Striker2}. Since the work of Cameron and Fon-der-Flaass, researchers in the growing field of dynamic algebraic combinatorics have studied interesting toggle groups and related concepts \cite{Chan, Many, Einstein, Einstein2, Grinberg, Grinberg2}. So-called \emph{Coxeter elements} have received special attention. A Coxeter element of the toggle group of an $m$-element poset $E$ is an element of the form $t_{e_1}t_{e_2}\cdots t_{e_m}$, where $e_1e_2\cdots e_m\in S_E$ is a permutation of the elements of $E$. If $e_1e_2\cdots e_m$ is a linear extension of $E$, then the Coxeter element $t_{e_1}t_{e_2}\cdots t_{e_m}$ is a special operation known as \emph{rowmotion}. Cameron and Fon-der-Flaass defined rowmotion (as with toggles, the name ``rowmotion" had to wait for Striker and Williams \cite{Striker2}) and showed that it does not depend on the choice of the linear extension $e_1e_2\cdots e_m$. Another important operation known as \emph{promotion} is defined by a different Coxeter element of the toggle group of a poset \cite{Stanley}. 

Recently, Striker \cite{Striker} defined generalized toggles and generalized toggle groups, proving structure theorems for certain generalized toggle groups and collecting several examples of actions that are elements of these groups. Roby's delightful survey article \cite{Roby} discusses connections between generalized toggle groups and an interesting property of certain group actions called \emph{homomesy}. 

Our motivation for mentioning these notions stems from an observation that the SDS maps we consider in this article resemble elements of generalized toggle groups; they are not quite the same, however. We will need the following more flexible notion of toggling. In the following definition, let $A\Delta B$ denote the symmetric difference of the sets $A$ and $B$. 

\begin{definition}\label{Def4}
Suppose $E$ is a set and $\mathscr L$ is a fixed collection of subsets of $E$. If $e\in E$, then a \emph{flexible toggle of} $E$ \emph{at} $e$ is a function $\mathscr T_e\colon \mathscr L\to \mathscr L$ satisfying 
\begin{enumerate}[(i)]
\item $\mathscr T_e^2(X)=X$
\item $X\Delta \mathscr T_e(X)\subseteq\{e\}$
\end{enumerate} for all $X\in \mathscr L$. If, for each $e\in E$, we fix $\mathscr T_e$ to be a specific flexible toggle of $E$ at $e$, then the subgroup of the symmetric group $S_{\mathscr L}$ generated by the set $\{\mathscr T_e\colon e\in E\}$ is a \emph{flexible toggle group} of $E$. If $\vert E\vert=m<\infty$, then a \emph{Coxeter element} of this flexible toggle group is an element of the form $\mathscr T_{e_1}\mathscr T_{e_2}\cdots\mathscr T_{e_m}$, where $e_1e_2\cdots e_m\in S_E$.   
\end{definition} 

Consider a symmetric invertible SDS defined over a graph $G$ with vertex set $V=\{v_1,\ldots,v_n\}$. Suppose the set of states of the SDS is $\mathbb F_2$, and let $\{f_{v_i}\}_{i=1}^n$ be the collection of vertex functions. According to Proposition \ref{Thm1}, each vertex function $f_{v_i}$ is either $\text{parity}_{\deg(v_i)+1}$ or $\text{(1+parity)}_{\deg(v_i)+1}$. Therefore, each local function $F_{v_i}\colon\mathbb F_2^n\to\mathbb F_2^n$ is an involution. There is a natural correspondence between the system states of the SDS and the subsets of $V$; the system state $(q_1,\ldots,q_n)$ corresponds to the set $\{v_i\in V\colon q_i=1\}$. If we use this correspondence to view each local function $F_{v_i}$ as a map from $2^{V}$ to $2^{V}$, then $F_{v_i}$ is actually a flexible toggle of $V$ at $v_i$ (with respect to the fixed set $\mathscr L=2^V$). Let $H$ be the flexible toggle group of $V$ generated by $\{F_{v_1},\ldots,F_{v_n}\}$. The Coxeter elements of $H$ are precisely the SDS maps of the form $[G,\{f_{v_i}\}_{i=1}^n,\pi]$ for $\pi\in S_n$. 

The preceding paragraph suggests that one might be able to study certain flexible toggle groups by recasting them in the language of SDS. Results from the paper \cite{Macauley2} give upper bounds on the number of cyclically nonequivalent SDS maps one can obtain by varying the update orders of certain sequential dynamical systems; these results then yield upper bounds for the number of dynamically nonequivalent (these SDS maps are bijective, so cyclical equivalence is the same as dynamical equivalence) Coxeter elements in certain flexible toggle groups.  

\section{Dynamics of Affine Transformations}
An affine transformation of a vector space $V$ is a map $T\colon V\to V$ given by $T(v)=L(v)+b$ for some linear map $L\colon V\to V$ and some fixed $b\in V$. In this section, we review some results and terminology from linear algebra and discuss previous work of Elspas \cite{Elspas} and Milligan and Wilson \cite{Milligan} concerning the dynamics of affine transformations of $\mathbb F_2$-vector spaces. 

If $R$ is a commutative ring and $M$ is an $R$-module, we say an element $a$ of $R$ \emph{annihilates} an element $z$ of $M$ if $az=0$. We say $a$ annihilates a subset $Z$ of $M$ if $az=0$ for all $z\in Z$. Suppose $V$ is a finite-dimensional vector space over a field $F$ and that $L\colon V\to V$ is a fixed linear transformation. The linear map $L$ makes $V$ into a module over the polynomial ring $F[x]$. The action of a polynomial $p(x)=a_kx^k+\cdots+a_1x+a_0$ on $v\in V$ is given by $p(x)v=a_kL^k(v)+\cdots+a_1L(v)+a_0v$. For any $p(x)\in F[x]$, the set of vectors in $V$ that $p(x)$ annihilates forms a vector subspace of $V$. We say that $L$ \emph{satisfies} $p(x)$ if $p(x)$ annihilates $V$. There is a unique monic polynomial of minimum degree that $L$ satisfies. This polynomial, denoted $m_L(x)$, is called the \emph{minimal polynomial} of $L$. A polynomial $p(x)\in F[x]$ is divisible by $m_L(x)$ if and only if $L$ satisfies $p(x)$. 

Let $I_m$ denote the $m\times m$ identity matrix over a field $F$. The \emph{characteristic polynomial} of an $m\times m$ matrix $M$ is the polynomial $\det(xI_m-M)$. If $M$ is a matrix of $L$ with respect to some basis for $V$, then the characteristic polynomial of $L$, denoted $\chi_L(x)$, is simply the characteristic polynomial of $M$. The Cayley-Hamilton theorem states that $L$ satisfies $\chi_L(x)$. Equivalently, $m_L(x)$ divides $\chi_L(x)$ in $F[x]$. 

For each fixed $L$, there exist monic polynomials $\alpha_1(x),\ldots,\alpha_t(x)$ such that $\alpha_1(x)\mid\cdots\mid\alpha_t(x)$ and \[V\cong \bigoplus_{i=1}^tF[x]/(\alpha_i(x))\] (this is an isomorphism of $F[x]$-modules). Moreover, $\prod_{i=1}^t\alpha_i(x)=\chi_L(x)$ and $\alpha_t(x)=m_L(x)$. It follows that if $p(x)$ is a separable polynomial (a polynomial with no repeated factors) that divides $\chi_L(x)$, then $p(x)$ divides $m_L(x)$. The polynomials $\alpha_i(x)$ are called the \emph{invariant factors} of $L$, and the sequence $\alpha_1(x),\ldots,\alpha_t(x)$ is the \emph{invariant factor decomposition} of $L$. If $\alpha_i(x)=q_{i,1}(x)^{\gamma_{i,1}}\cdots q_{i,\kappa_i}(x)^{\gamma_{i,\kappa_i}}$ is the factorization of $\alpha_i(x)$ into a product of powers of monic irreducible polynomials, then the polynomials $q_{i,j}(x)^{\gamma_{i,j}}$ are the \emph{elementary divisors} of $L$. 

In \cite{Elspas}, Elspas showed how to completely determine the dynamics of the linear transformation $L$ from the list of elementary divisors of $L$ provided $F$ is a finite field. Elspas' work is neatly summarized and extended in \cite{Toledo}. Milligan and Wilson built upon Elspas' results to determine the dynamics of an affine transformation $T\colon \mathbb F_2^n\to\mathbb F_2^n$. Their results also require knowledge of the elementary divisors of a certain linear map $\widetilde T$ obtained from $T$ (see Definition \ref{Def3} below for the definition of $\widetilde T$). In contrast, our method only requires knowledge of the invariant factors of these linear transformations. This method makes heavy use of the following proposition. 

\begin{proposition}\label{Prop1}
Let $V$ be a finite-dimensional vector space over a field $F$. Let $L\colon V\to V$ be a linear transformation whose invariant factor decomposition is $\alpha_1(x),\ldots,\alpha_t(x)$. For each positive integer $r$, the set $\text{\emph{Per}}_1(L^r)$ of fixed points of $L^r$ is a vector subspace of $V$ of dimension \[\sum_{i=1}^t\deg(\gcd(x^r-1,\alpha_i(x))).\]
\end{proposition}       
\begin{proof}
Fix a positive integer $r$. It is clear that $\text{Per}_1(L^r)$ is a vector subspace of $V$ since $L^r$ is a linear map. In fact, if we make $V$ into an $F[x]$-module via the linear transformation $L$, then $\text{Per}_1(L^r)$ is precisely the subspace of all vectors annihilated by the polynomial $x^r-1$. There is an $F[x]$-module isomorphism $h\colon V\to \bigoplus_{i=1}^tF[x]/(\alpha_i(x))$, and $h(\text{Per}_1(L^r))$ is the set of elements of $\bigoplus_{i=1}^tF[x]/(\alpha_i(x))$ annihilated by $x^r-1$. Let $Y_i$ denote the set of elements of the $F[x]$-module $F[x]/(\alpha_i(x))$ annihilated by $x^r-1$. We have $h(\text{Per}_1(L^r))=\bigoplus_{i=1}^tY_i$, so it suffices to show that $Y_i$ is an $F$-vector space of dimension $\deg(\gcd(x^r-1,\alpha_i(x)))$. 

Let $Q_i(x)=\gcd(x^r-1,\alpha_i(x))$. A coset $p(x)+(\alpha_i(x))$ is an element of $Y_i$ if and only if $\alpha_i(x)$ divides $(x^r-1)p(x)$. This occurs if and only if $\alpha_i(x)/Q_i(x)$ divides $p(x)$. As a consequence, one may easily verify that the map $g\colon Y_i\to F[x]/(Q_i(x))$ given by \[g(p(x)+(\alpha_i(x)))=p(x)Q_i(x)/\alpha_i(x)+(Q_i(x))\] is a well-defined isomorphism of $F$-vector spaces. The desired result follows since the dimension of $F[x]/(Q_i(x))$ is $\deg(Q_i(x))$.    
\end{proof}

Proposition \ref{Prop1} and the classical number-theoretic M\"obius inversion formula allow us to determine the dynamics of any vector space automorphism of $\mathbb F_2^n$ provided we know the invariant factor decomposition of the automorphism. Suppose, however, that we wish to study the dynamics of a bijective affine transformation $T\colon \mathbb F_2^n\to\mathbb F_2^n$. Following \cite{Milligan}, we make the following definition. 

\begin{definition}\label{Def3}
If $M$ is an $n\times n$ matrix with entries in $\mathbb F_2$ and $b=(b_1,\ldots,b_n)\in\mathbb F_2^n$, let \[\widetilde M_b=\left( \begin{array}{cccc}
1 & 0 & \cdots & 0 \\
b_1 &   &   &   \\
\vdots &   & M &   \\
b_n &   &   &   \end{array} \right).\]
Suppose $T\colon\mathbb F_2^n\to\mathbb F_2^n$ is an affine transformation given by $T(v)=L(v)+b$, where $L\colon\mathbb F_2^n\to\mathbb F_2^n$ is a linear map whose matrix with respect to the standard basis for $\mathbb F_2^n$ is $M$. Define $\widetilde T\colon\mathbb F_2^{n+1}\to\mathbb F_2^{n+1}$ to be the linear map whose matrix with respect to the standard basis for $\mathbb F_2^{n+1}$ is $\widetilde M_b$. 
\end{definition}

For $i\in\mathbb F_2$, put $Z_i=\{(z_0,z_1,\ldots,z_n)\in\mathbb F_2^{n+1} : z_0=i\}$, and let $P_i\colon Z_i\to\mathbb F_2^n$ be the bijection given by $P_i(i,x_1,\ldots,x_n)=(x_1,\ldots,x_n)$. Preserve the notation from Definition \ref{Def3}. Each set $Z_i$ is invariant under $\widetilde T$ (meaning $\widetilde T(Z_i)\subseteq Z_i$). Furthermore, $\widetilde T\vert_{Z_0}\colon Z_0\to Z_0$ and $\widetilde T\vert_{Z_1}\colon Z_1\to Z_1$ are dynamically equivalent to $L$ and $T$, respectively. Indeed, $L=P_0\circ\widetilde T\vert_{Z_0}\circ P_0^{-1}$ and $T=P_1\circ\widetilde T\vert_{Z_1}\circ P_1^{-1}$. Consequently, we can deduce the dynamics of $T$ from the dynamics of $\widetilde T$ and $L$. More precisely,  
\begin{equation}\label{Eq4}
\vert\text{Per}_1(T^r)\vert=\vert\text{Per}_1(\widetilde T^r)\vert-\vert\text{Per}_1(L^r)\vert \hspace{0.5cm}\text{and}\hspace{0.5cm}\vert\text{Per}_r(T)\vert=\vert\text{Per}_r(\widetilde T)\vert-\vert\text{Per}_r(L)\vert
\end{equation} 
for all positive integers $r$.

Many of the results in \cite{Milligan} relate the invariant factors and the dynamics of $L$ to those of $\widetilde T$. The following proposition summarizes those results that we need. Note that part of this proposition follows from the proof of Lemma 4 in \cite{Milligan} rather than the lemma itself.  

\begin{proposition}[\cite{Milligan}, Lemmas 2 and 4]\label{Prop2}
Preserve the notation from Definition \ref{Def3}. Let $\alpha_1(x),$ $\ldots,\alpha_t(x)$ be the invariant factor decomposition of $L$, and let $\beta_1(x),\ldots,\beta_u(x)$ be the invariant factor decomposition of $\widetilde T$. The maps $L$ and $T$ are cyclically equivalent if and only if $T$ has a fixed point. If $T$ has a fixed point, then $u=t+1$, $\beta_1(x)=x+1$, and $\alpha_i(x)=\beta_{i+1}(x)$ whenever $1\leq i\leq t$. If $T$ does not have a fixed point, then $u=t$. In this case, there exists $j\in\{1,\ldots,t\}$ such that $\beta_j(x)=(x+1)\alpha_j(x)$ and $\alpha_i(x)=\beta_i(x)$ whenever $i\neq j$. 
\end{proposition} 

If $T$ has a fixed point $a$, then $L$ and $T$ are actually dynamically equivalent. The map $\eta\colon\mathbb F_2^n\to\mathbb F_2^n$ given by $\eta(z)=z+a$ is a bijection such that $T\circ\eta=\eta\circ L$. Indeed, for any $z\in\mathbb F_2^n$, \[T(\eta(z))=T(z+a)=L(z)+L(a)+b=L(z)+T(a)=L(z)+a=\eta(L(z)).\] 

The equation $T(a)=a$ is equivalent to $b=L(a)+a$. Therefore, the affine transformation $T$ has a fixed point if and only if $b$ is an element of the subspace $\{L(z)+z : z\in\mathbb F_2^n\}$ of $\mathbb F_2^n$. This observation shows that it is often not too difficult to see whether or not $T$ has a fixed point, which illustrates why Proposition \ref{Prop2} is so powerful.    

\section{Dynamics of Symmetric Invertible Asynchronous Elementary Cellular Automata} 
Throughout this section, fix $T$ to be the SDS map of an SIAECA with the identity update order. Equivalently, $T=[C_n,\{f_{v_i}\}_{i=1}^n,\text{id}]\colon\mathbb F_2^n\to\mathbb F_2^n$, where $f_{v_i}\in\{\text{parity}_3,(1+\text{parity})_3\}$ for all $1\leq i\leq n$. To ease notation, let $L_n=[C_n,\text{parity}_3,\text{id}]$. The following lemma, whose straightforward proof we omit, illustrates why we needed to discuss the dynamics of affine linear transformations in the previous section. 
\begin{lemma}\label{Lem1}
The map $L_n\colon\mathbb F_2^n\to\mathbb F_2^n$ is an $\mathbb F_2$-linear map given explicitly by \[L_n(y_1,\ldots,y_n)=(y_1+y_2+y_n,y_1+y_3+y_n,y_1+y_4+y_n,\ldots,y_1+y_{n-1}+y_n,y_1,y_2).\] The map $T\colon\mathbb F_2^n\to\mathbb F_2^n$ is an affine transformation given by $T(v)=L_n(v)+b$ for some fixed $b\in\mathbb F_2^n$.
\end{lemma}

In order to study the dynamics of the affine map $T$, we first analyze the linear map $L_n$. Let $B_n$ denote the matrix of $L_n$ with respect to the standard basis for $\mathbb F_2^n$. Keeping in mind that the entries in $B_n$ are elements of $\mathbb F_2$, we can compute the characteristic polynomial of $L_n$ as follows. First, add the first column of $B_n+xI_n$ to the last column of $B_n+xI_n$. Next, add the last row of the resulting matrix to its first row. Let $D_n(x)$ be the matrix obtained after performing these two operations. 
For example, in the case $n=4$, we have \[B_4+xI_4=\left( \begin{array}{cccc}
x+1 & 1 & 0 & 1 \\
1 & x & 1 & 1 \\
1 & 0 & x & 0 \\
0 & 1 & 0 & x \end{array} \right)\hspace{0.5cm}\text{and}\hspace{0.5cm}D_4(x)=\left( \begin{array}{cccc}
x+1 & 0 & 0 & 0 \\
1 & x & 1 & 0 \\
1 & 0 & x & 1 \\
0 & 1 & 0 & x \end{array} \right).\]

The first row of $D_n(x)$ is the $n$-tuple whose first entry is $x+1$ and whose other entries are all $0$'s. Let $I_m'$ denote the $m\times m$ matrix whose $(i,j)$-entry is $1$ if $j\equiv i+1\pmod{m}$ and is $0$ otherwise. The matrix obtained by deleting the first row and the first column of $D_n(x)$ is $I_{n-1}'+xI_{n-1}$. Therefore, \[\chi_{L_n}(x)=\det(B_n+xI_n)=\det(D_n(x))=(x+1)\det(I_{n-1}'+xI_{n-1}).\] Using cofactor expansion along the first column, we easily find that the determinant of $I_{n-1}'+xI_{n-1}$ is $x^{n-1}-1$. Consequently, 
\begin{equation}\label{Eq2}
\chi_{L_n}(x)=(x+1)(x^{n-1}-1).
\end{equation}    

To determine the invariant factor decomposition of $L_n$, we need to know a bit more information about the dynamics of $L_n$. To obtain this information, we follow Chapter 5 of \cite{Mortveit} and make the following definitions. 
\begin{definition} \label{Def1}
Define a map $\iota_n\colon\mathbb F_2^n\rightarrow\mathbb F_2^{2n-2}$ by \[\iota_n(a_1,\ldots,a_n)=(a_1,\ldots,a_n,a_1+a_2+a_n,a_1+a_3+a_n,\ldots,a_1+a_{n-1}+a_n).\] Let $\widehat{\mathbb F}_2^{2n-2}=\iota_n(\mathbb F_2^n)$ be the image of $\mathbb F_2^n$ under $\iota_n$. 
\end{definition} 
\begin{definition} \label{Def2} 
Define a shift map $\sigma_{2n-2}\colon\mathbb F_2^{2n-2}\rightarrow\mathbb F_2^{2n-2}$ by \[\sigma_{2n-2}(a_1,\ldots,a_{2n-2})=(a_{n+1},a_{n+2},\ldots,a_{2n-2},a_1,\ldots,a_n).\] Let $\overline\sigma_{2n-2}\colon\widehat{\mathbb F}_2^{2n-2}\rightarrow\widehat{\mathbb F}_2^{2n-2}$ be the restriction of $\sigma_{2n-2}$ to $\widehat{\mathbb F}_2^{2n-2}$.
\end{definition}
One may easily verify that if $z\in\widehat{\mathbb F}_2^{2n-2}$, then $\sigma_{2n-2}(z)\in\widehat{\mathbb F}_2^{2n-2}$. This means that the range of $\overline\sigma_{2n-2}$ is indeed contained in $\widehat{\mathbb F}_2^{2n-2}$, so it is valid to write $\overline\sigma_{2n-2}\colon\widehat{\mathbb F}_2^{2n-2}\rightarrow\widehat{\mathbb F}_2^{2n-2}$. 
 
The following fact, whose proof appears within the proof of Proposition 5.23 in \cite{Mortveit}, provides useful information about the dynamics of $L_n$.  
\begin{fact} \label{Thm2}
The map $\iota_n\colon\mathbb F_2^n\rightarrow\widehat{\mathbb F}_2^{2n-2}$ is a vector space isomorphism such that $\iota_n\circ L_n=\overline{\sigma}_{2n-2}\circ\iota_n$. Thus, the maps $L_n$ and $\overline\sigma_{2n-2}$ are dynamically equivalent. 
\end{fact}

The following corollary to Fact \ref{Thm2} is essentially Proposition 5.23\footnote{The statement of Proposition 5.23 in \cite{Mortveit} is slightly mistaken. However, it is clear from the proof how to amend the proposition, and we have done so in the statement of Corollary \ref{Cor1}.} in \cite{Mortveit}. 
\begin{corollary} \label{Cor1} 
If $L_n$ has a periodic point of period $r$, then $r\mid 2n-2$. If, in addition, $n$ is even, then $r\mid n-1$.  
\end{corollary} 

\begin{lemma}\label{Lem2}
The invariant factor decomposition of $L_n$ is $x+1,x^{n-1}-1$ if $n$ is even and \\ $(x+1)(x^{n-1}-1)$ if $n$ is odd. In particular, $m_{L_n}(x)=\chi_{L_n}(x)$ if $n$ is odd. 
\end{lemma}
\begin{proof}
First, assume $n$ is even. The derivative of $x^{n-1}-1$ is $(n-1)x^{n-2}=x^{n-2}$, which is relatively prime to $x^{n-1}-1$. It follows that $x^{n-1}-1$ is a separable polynomial that divides $\chi_{L_n}(x)$, so $x^{n-1}-1$ divides $m_{L_n}(x)$. Corollary \ref{Cor1} tells us that $L_n$ satisfies the polynomial $x^{n-1}-1$, so $m_{L_n}(x)=x^{n-1}-1$. Because $\chi_{L_n}(x)=(x+1)(x^{n-1}-1)$, the invariant factor decomposition of $L_n$ is $x+1,x^{n-1}-1$. 

Next, assume $n$ is odd. For each $j\in\{0,1,\ldots,n-1\}$, let $w_j$ be the $(2n-2)$-tuple of $0$'s and $1$'s whose $i^{\text{th}}$ entry is $1$ if and only if $j+1\leq i\leq j+n-1$. Let $z$ be the $n$-tuple whose $n^\text{th}$ entry is $0$ and whose other entries are all $1$'s. Observe that $w_0=\iota_n(z)$. Since $n$ is relatively prime to $2n-2$, the forward orbit of $w_0$ under $\overline\sigma_{2n-2}$ contains all $2n-2$ cyclic shifts of the tuple $w_0$. In particular, $w_0,w_1,\ldots,w_{n-1}$ are all in the forward orbit of $w_0$ under $\overline{\sigma}_{2n-2}$. As a consequence, $w_0,w_1,\ldots,w_{n-1}\in\widehat{\mathbb F}_2^{2n-2}$. The vectors $w_0,w_1,\ldots,w_{n-1}$ are linearly independent, so they form a basis for the $n$-dimensional space $\widehat{\mathbb F}_2^{2n-2}$. This shows that the forward orbit of $w_0$ under $\overline{\sigma}_{2n-2}$ generates the vector space $\widehat{\mathbb F}_2^{2n-2}$, so it follows from Fact \ref{Thm2} that the forward orbit of $z$ under $L_n$ generates $\mathbb F_2^n$. This means that $L_n$ cannot satisfy a polynomial in $\mathbb F_2[x]$ of degree less than $n$, so the degree of $m_{L_n}(x)$ is at least $n$. Thus, $m_{L_n(x)}=\chi_{L_n}(x)=(x+1)(x^{n-1}-1)$.   
\end{proof}

If $k$ and $m$ are nonzero integers, we write $k\mid_o m$ if $m/k$ is an odd integer. Recall that we fixed an SIAECA with an SDS map $T$ at the beginning of this section. When calculating $\vert\text{Per}_1(T^r)\vert$ for positive integers $r$, we must pay special attention to the case in which $r\mid_o 2n-2$. The following lemma will allow us to use Proposition \ref{Prop1} to understand this case.  

\begin{lemma}\label{Lem3}
Let $r$ be a positive integer such that $r\mid_o 2n-2$. In $\mathbb F_2[x]$, we have 
\begin{enumerate}[(i)]
\item $\gcd(x^r-1,x^{n-1}-1)=x^{r/2}-1$ if $n$ is even;
\item $\gcd(x^r-1,(x+1)(x^{n-1}-1))=(x+1)(x^{r/2}-1)$;
\item $\gcd(x^r-1,(x+1)^2(x^{n-1}-1))=(x+1)^2(x^{r/2}-1)$ if $n$ is odd. 
\end{enumerate}
\end{lemma}
\begin{proof}
Let $q_1(x)=\dfrac{x^r-1}{(x+1)(x^{r/2}-1)}=\dfrac{x^{r/2}-1}{x+1}$ and $q_2(x)=\dfrac{(x+1)(x^{n-1}-1)}{(x+1)(x^{r/2}-1)}=\dfrac{x^{n-1}-1}{x^{r/2}-1}$. Note that $q_1(x)\in\mathbb F_2[x]$. Moreover, $q_2(x)\in\mathbb F_2[x]$ because $r/2$ divides $n-1$. In order to prove $(ii)$, it suffices to show that $q_1(x)$ and $q_2(x)$ are relatively prime. Because $r\mid_o 2n-2$, we may write $n-1=rs/2$ for some odd integer $s$. We have \[q_2(x)=1+x^{r/2}+(x^{r/2})^2+\cdots+(x^{r/2})^{s-1}=(1+x^{r/2})(1+x^r+x^{2r}+\cdots+x^{((s-3)/2)r})+(x^{r/2})^{s-1}\] \[=q_1(x)(x+1)(1+x^r+x^{2r}+\cdots+x^{((s-3)/2)r})+(x^{r/2})^{s-1}.\] It follows that the greatest common divisor of $q_1(x)$ and $q_2(x)$ divides $(x^{r/2})^{s-1}$, so it must be $1$. 

From $(ii)$, we deduce that $\gcd(x^r-1,x^{n-1}-1)$ is either $x^{r/2}-1$ or $(x+1)(x^{r/2}-1)$. Suppose $n$ is even. The polynomial $x^{n-1}-1$ has no repeated factors because it is relatively prime to its derivative. Because $(x+1)^2$ divides $(x+1)(x^{r/2}-1)$, it follows that $(x+1)(x^{r/2}-1)$ cannot be the greatest common divisor of $x^r-1$ and $x^{n-1}-1$. This proves $(i)$. 

Suppose $n$ is odd. Using $(ii)$ again, we see that $\gcd(x^r-1,(x+1)^2(x^{n-1}-1))$ is either \\$(x+1)(x^{r/2}-1)$ or $(x+1)^2(x^{r/2}-1)$. To prove $(iii)$, it suffices to show that $x^r-1$ is divisible by $(x+1)^2(x^{r/2}-1)$. Equivalently, we must show that $(x+1)^2$ divides $x^{r/2}-1$. Because $n$ is odd and $r\mid_o 2n-2$, $r$ is a multiple of $4$. This shows that $r/2$ is even, so $x^{r/2}-1$ is divisible by $x^2-1=(x+1)^2$.   
\end{proof}

We are now in a position to compute $\vert\text{Per}_1(T^r)\vert$. By Proposition \ref{Prop2}, $m_{\widetilde T}(x)$ divides $(x+1)m_{L_n}(x)$. If $n$ is even, then it follows from Lemma \ref{Lem2} that $m_{\widetilde T}(x)$ divides $(x+1)(x^{n-1}-1)$, which in turn divides $x^{2n-2}-1$. Similarly, if $n$ is odd, then $m_{\widetilde T}(x)$ divides $(x+1)^2(x^{n-1}-1)$, which divides $x^{2n-2}-1$. In either case, $m_{\widetilde T}(x)$ divides $x^{2n-2}-1$. This shows that $\text{Per}_1(\widetilde T^{2n-2})=\mathbb F_2^{n+1}$. According to \eqref{Eq4}, $\text{Per}_1(T^{2n-2})=\mathbb F_2^n$. We find that $\text{Per}_1(T^r)=\text{Per}_1(T^{\gcd(r,2n-2)})$ for all positive integers $r$. Thus, when computing $\vert\text{Per}_1(T^r)\vert$, it suffices to consider the case in which $r\mid 2n-2$. 

\begin{lemma}\label{Lem4}
Let $r$ be a positive integer such that $r\mid 2n-2$. Let $T=[C_n,\{f_{v_i}\}_{i=1}^n,\text{\emph{id}}]$, where $f_{v_i}\in\{\text{\emph{parity}}_3,(1+\text{\emph{parity}})_3\}$ for all $i$. If $T$ has a fixed point, then \[\vert\text{\emph{Per}}_1(T^r)\vert=\begin{cases} 2^{r+1}, & \mbox{if } r\mid n-1\text{ and }n\equiv 0\pmod 2; \\ 2^r, & \mbox{if } r\mid n-1\text{ and }n\equiv 1\pmod 2; \\ 2^{\frac r2+1}, & \mbox{if } r\nmid n-1. \end{cases}\] If $T$ does not have a fixed point, then \[\vert\text{\emph{Per}}_1(T^r)\vert=\begin{cases} 0, & \mbox{if } r\mid n-1; \\ 2^{\frac r2+1}, & \mbox{if } r\nmid n-1. \end{cases}\] 
\end{lemma} 
\begin{proof}
Suppose first that $T$ has a fixed point. We know by Proposition \ref{Prop2} that $T$ and $L_n$ are cyclically equivalent, so $\vert\text{Per}_1(T^r)\vert=\vert\text{Per}_1(L_n^r)\vert$. Recall from Lemma \ref{Lem2} that the invariant factor decomposition of $L_n$ is $x+1,x^{n-1}-1$ if $n$ is even and $(x+1)(x^{n-1}-1)$ if $n$ is odd. If $r\mid n-1$ and $n$ is even, Proposition \ref{Prop1} tells us that $\text{Per}_1(L_n^r)$ is a vector subspace of $\mathbb F_2^n$ of dimension \[\deg(\gcd(x^r-1,x+1))+\deg(\gcd(x^r-1,x^{n-1}-1))=\deg(x+1)+\deg(x^r-1)=r+1.\] Similarly, if $r\mid n-1$ and $n$ is odd, then \[\dim (\text{Per}_1(L_n^r))=\deg(\gcd(x^r-1,(x+1)(x^{n-1}-1)))=\deg(x^r-1)=r.\] Now suppose $r\nmid n-1$. Equivalently, $r\mid_o 2n-2$. If $n$ is even, then we know from Lemma \ref{Lem3} that the dimension of $\text{Per}_1(L_n^r)$ is \[\deg(\gcd(x^r-1,x+1))+\deg(\gcd(x^r-1,x^{n-1}-1))=\deg(x+1)+\deg(x^{r/2}-1)=\frac r2+1.\] If $n$ is odd, then we may use Lemma \ref{Lem3} once again to find that \[\dim (\text{Per}_1(L_n^r))=\deg(\gcd(x^r-1,(x+1)(x^{n-1}-1)))=\deg((x+1)(x^{r/2}-1))=\frac r2+1.\] 

Now assume that $T$ does not have a fixed point and that $n$ is even. According to Proposition \ref{Prop2}, the invariant factor decomposition of $\widetilde T$ is either $(x+1)^2,x^{n-1}-1$ or $x+1,(x+1)(x^{n-1}-1)$. The first case is impossible since $(x+1)^2$ does not divide $x^{n-1}-1$ ($x^{n-1}-1$ has no repeated factors because $n$ is even). It follows that the invariant factor decomposition of $\widetilde T$ is $x+1,(x+1)(x^{n-1}-1)$, so we may invoke Proposition \ref{Prop1} to find that $\text{Per}_1(\widetilde T^r)$ is a vector subspace of $\mathbb F_2^{n+1}$ of dimension \[\deg(\gcd(x^r-1,x+1))+\deg(\gcd(x^r-1,(x+1)(x^{n-1}-1)))=1+\deg(\gcd(x^r-1,(x+1)(x^{n-1}-1))).\] If $r\mid n-1$, then $\gcd(x^r-1,(x+1)(x^{n-1}-1))=x^r-1$. In this case, we know from \eqref{Eq4} and our calculations above that \[\vert\text{Per}_1(T^r)\vert=\vert\text{Per}_1(\widetilde T^r)\vert-\vert\text{Per}_1(L_n^r)\vert=2^{r+1}-2^{r+1}=0.\] Suppose $r\nmid n-1$. Then $\gcd(x^r-1,(x+1)(x^{n-1}-1))=(x+1)(x^{r/2}-1)$ by Lemma \ref{Lem3}, so $\text{Per}_1(\widetilde T^r)$ has dimension $\frac r2+2$. We saw above that $\text{Per}_1(L_n^r)$ is a vector subspace of $\mathbb F_2^n$ of dimension $\frac r2+1$. By \eqref{Eq4}, \[\vert\text{Per}_1(T^r)\vert=\vert\text{Per}_1(\widetilde T^r)\vert-\vert\text{Per}_1(L_n^r)\vert=2^{\frac r2+2}-2^{\frac r2+1}=2^{\frac r2+1}.\]

Finally, assume $T$ has no fixed points and $n$ is odd. Proposition \ref{Prop2} tells us that the invariant factor decomposition of $\widetilde T$ is $(x+1)^2(x^{n-1}-1)$ (so $m_{\widetilde T}(x)=\chi_{\widetilde T}(x)$). By Proposition \ref{Prop1}, $\text{Per}_1(\widetilde T^r)$ is a subspace of $\mathbb F_2^{n+1}$ of dimension $\deg(\gcd(x^r-1,(x+1)^2(x^{n-1}-1)))$. If $r\mid n-1$, then \\$\gcd(x^r-1,(x+1)^2(x^{n-1}-1))=x^r-1$. In this case, we also know from above that $\text{Per}_1(L_n^r)$ is a vector space over $\mathbb F_2$ of dimension $r$. Therefore, \[\vert\text{Per}_1(T^r)\vert=\vert\text{Per}_1(\widetilde T^r)\vert-\vert\text{Per}_1(L_n^r)\vert=2^r-2^r=0.\] Next, suppose $r\nmid n-1$. By Lemma \ref{Lem3}, $\gcd(x^r-1,(x+1)^2(x^{n-1}-1))=(x+1)^2(x^{r/2}-1)$. We saw above that $\text{Per}_1(L_n^r)$ has dimension $\frac r2+1$ in this case, so \[\vert\text{Per}_1(T^r)\vert=\vert\text{Per}_1(\widetilde T^r)\vert-\vert\text{Per}_1(L_n^r)\vert=2^{\frac r2+2}-2^{\frac r2+1}=2^{\frac r2+1}. \qedhere\] 
\end{proof}

In what follows, let $\mu$ denote the number-theoretic M\"obius function. Recall that for positive integers $k$ and $m$, we write $k\mid_o m$ to mean that $m/k$ is an odd integer. In addition, let \[\xi_m(k)=\begin{cases} 2^k, & \mbox{if } k\mid m-1; \\ 2^{\frac k2+1}, & \mbox{if } k\nmid m-1. \end{cases}\]
\begin{theorem} \label{Thm3} 
Let $r$ be a positive integer. Let $T=[C_n,\{f_{v_i}\}_{i=1}^n,\text{\emph{id}}]$, where \\$f_{v_i}\in\{\text{\emph{parity}}_3,(1+\text{\emph{parity}})_3\}$ for all $i$. If $T$ has a fixed point, then \[\vert\text{\emph{Per}}_r(T)\vert=\begin{cases} \sum\limits_{d\mid r}\mu(r/d)2^{d+1}, & \mbox{if } r\mid n-1\text{ and }n\equiv 0\!\pmod 2; \\ \sum\limits_{d\mid r}\mu(r/d)\xi_n(d), & \mbox{if } r\mid 2n-2\text{ and }n\equiv 1\!\pmod 2; \\ 0, & \mbox{otherwise}. \end{cases}\]
If $T$ does not have a fixed point, then \[\vert\text{\emph{Per}}_r(T)\vert=\begin{cases} \sum\limits_{d\mid_or}\mu(r/d)2^{\frac d2+1}, & \mbox{if } r\mid_o2n-2; \\ 0, & \mbox{otherwise}. \end{cases}\]
\end{theorem}
\begin{proof}
In the paragraph preceding Lemma \ref{Lem4}, we saw that $\text{Per}_1(T^{2n-2})=\mathbb F_2^n$. Therefore, if $r\nmid 2n-2$, then $T$ has no periodic points of period $r$. We also know that $L_n$ has no periodic points of period $r$ if $n$ is even and $r\nmid n-1$. Consequently, if $T$ has a fixed point, $n$ is even, and $r\nmid n-1$, then $\vert\text{Per}_r(T)\vert=0$ because $T$ and $L_n$ are dynamically equivalent by Proposition \ref{Prop2}. The remaining cases follow from Lemma \ref{Lem4} and the formula in \eqref{Eq3}, which tells us that \[\vert\text{Per}_r(T)\vert=\sum_{d\mid r}\mu(r/d)\vert\text{Per}_1(T^r)\vert. \qedhere\] 
\end{proof}

\begin{table}[h]
{\footnotesize
\begin{tabu}{|c|[1.5pt]c|c|c|c|c|c|c|c|c|c|c|c|p{0.25cm}|p{0.25cm}|p{0.25cm}|c|c|c|c|c|}

    \hline \backslashbox{$n$}{$r$} & 1 & 2 & 3 & 4 & 5 & 6 & 7 & 8 & 9 & 10 & 11 & 12 & 13 & 14 & 15 & 16 & 17 & 18 & 19 & 20 \\\tabucline[2pt]{-}
    3 & 2 & 2 & 0 & 4 & 0 & 0 & 0 & 0 & 0 & 0 & 0 & 0 & 0 & 0 & 0 & 0 & 0 & 0 & 0 & 0 \\\hline
    4 & 4 & 0 & 12 & 0 & 0 & 0 & 0 & 0 & 0 & 0 & 0 & 0 & 0 & 0 & 0 & 0 & 0 & 0 & 0 & 0 \\\hline
    5 & 2 & 2 & 0 & 12 & 0 & 0 & 0 & 16 & 0 & 0 & 0 & 0 & 0 & 0 & 0 & 0 & 0 & 0 & 0 & 0 \\\hline
    6 & 4 & 0 & 0 & 0 & 60 & 0 & 0 & 0 & 0 & 0 & 0 & 0 & 0 & 0 & 0 & 0 & 0 & 0 & 0 & 0 \\\hline
    7 & 2 & 2 & 6 & 4 & 0 & 64 & 0 & 0 & 0 & 0 & 0 & 60 & 0 & 0 & 0 & 0 & 0 & 0 & 0 & 0 \\\hline
    8 & 4 & 0 & 0 & 0 & 0 & 0 & 252 & 0 & 0 & 0 & 0 & 0 & 0 & 0 & 0 & 0 & 0 & 0 & 0 & 0 \\\hline
    9 & 2 & 2 & 0 & 12 & 0 & 0 & 0 & 240 & 0 & 0 & 0 & 0 & 0 & 0 & 0 & 256 & 0 & 0 & 0 & 0 \\\hline
    10 & 4 & 0 & 12 & 0 & 0 & 0 & 0 & 0 & 1008 & 0 & 0 & 0 & 0 & 0 & 0 & 0 & 0 & 0 & 0 & 0 \\\hline
    11 & 2 & 2 & 0 & 4 & 30 & 0 & 0 & 0 & 0 & 990 & 0 & 0 & 0 & 0 & 0 & 0 & 0 & 0 & 0 & 1020 \\\hline
\end{tabu}}\vspace{.3cm}
\caption{The values of $\vert\text{Per}_r(T)\vert$ for $3\leq n\leq 11$ and $1\leq r\leq 20$, where $T$ is as in Theorem \ref{Thm3} and $T$ has a fixed point.}\label{Tab1}
\end{table}

\begin{table}[h]
{\footnotesize
\begin{tabu}{|c|[1.5pt]c|c|c|c|c|c|c|c|c|c|c|c|c|c|c|c|c|c|c|c|}

    \hline \backslashbox{$n$}{$r$} & 1 & 2 & 3 & 4 & 5 & 6 & 7 & 8 & 9 & 10 & 11 & 12 & 13 & 14 & 15 & 16 & 17 & 18 & 19 & 20 \\\tabucline[2pt]{-}
    3 & 0 & 0 & 0 & 8 & 0 & 0 & 0 & 0 & 0 & 0 & 0 & 0 & 0 & 0 & 0 & 0 & 0 & 0 & 0 & 0 \\\hline
    4 & 0 & 4 & 0 & 0 & 0 & 12 & 0 & 0 & 0 & 0 & 0 & 0 & 0 & 0 & 0 & 0 & 0 & 0 & 0 & 0 \\\hline
    5 & 0 & 0 & 0 & 0 & 0 & 0 & 0 & 32 & 0 & 0 & 0 & 0 & 0 & 0 & 0 & 0 & 0 & 0 & 0 & 0 \\\hline
    6 & 0 & 4 & 0 & 0 & 0 & 0 & 0 & 0 & 0 & 60 & 0 & 0 & 0 & 0 & 0 & 0 & 0 & 0 & 0 & 0 \\\hline
    7 & 0 & 0 & 0 & 8 & 0 & 0 & 0 & 0 & 0 & 0 & 0 & 120 & 0 & 0 & 0 & 0 & 0 & 0 & 0 & 0 \\\hline
    8 & 0 & 4 & 0 & 0 & 0 & 0 & 0 & 0 & 0 & 0 & 0 & 0 & 0 & 252 & 0 & 0 & 0 & 0 & 0 & 0 \\\hline
    9 & 0 & 0 & 0 & 0 & 0 & 0 & 0 & 0 & 0 & 0 & 0 & 0 & 0 & 0 & 0 & 512 & 0 & 0 & 0 & 0 \\\hline
    10 & 0 & 4 & 0 & 0 & 0 & 12 & 0 & 0 & 0 & 0 & 0 & 0 & 0 & 0 & 0 & 0 & 0 & 1008 & 0 & 0 \\\hline
    11 & 0 & 0 & 0 & 8 & 0 & 0 & 0 & 0 & 0 & 0 & 0 & 0 & 0 & 0 & 0 & 0 & 0 & 0 & 0 & 2040 \\\hline
    \end{tabu}}\vspace{.3cm}
\caption{The values of $\vert\text{Per}_r(T)\vert$ for $3\leq n\leq 11$ and $1\leq r\leq 20$, where $T$ is as in Theorem \ref{Thm3} and $T$ does \emph{not} have a fixed point.}\label{Tab1}
\end{table}

Let us say that the triple $(T,n,r)$ is \emph{nice} if either $T$ has a fixed point and $r$ divides \\$(2n-2)/\gcd(2,n)$ or $T$ does not have a fixed point and $r\mid_o2n-2$. The preceding theorem implies that if $T$ has a periodic point of period $r$, then $(T,n,r)$ is nice. The following corollary asserts that the converse is true as well. 

\begin{corollary}\label{Cor3}
Let $T$ and $r$ be as in Theorem \ref{Thm3}. If $T$ has a fixed point and $r$ divides \\$(2n-2)/\gcd(2,n)$, then $T$ has at least $2$ periodic points of period $r$. If $T$ does not have a fixed point and $r\mid_o2n-2$, then $T$ has at least $4$ periodic points of period $r$.
\end{corollary}
\begin{proof}
The proof follows from Theorem \ref{Thm3} along with the inequalities \[\sum_{d\mid r}\mu(r/d)2^{d+1}=2^{r+1}+\sum_{\substack{d\mid r\\ d<r}}\mu(r/d)2^{d+1}\geq 2^{r+1}-\sum_{\substack{d\mid r \\ d<r}}2^{d+1}\geq 2^{r+1}-\sum_{k=1}^{r-1}2^{k+1}=4,\] \[\sum_{d\mid r}\mu(r/d)\xi_n(d)=\xi_n(r)+\sum_{\substack{d\mid r\\ d<r}}\mu(r/d)\xi_n(d)\geq2^{\left\lfloor r/2\right\rfloor+1}+\sum_{\substack{d\mid r\\ d<r}}\mu(r/d)\xi_n(d)\] \[\geq 2^{\left\lfloor r/2\right\rfloor+1}-\sum_{\substack{d\mid r \\ d<r}}\xi_n(d)\geq 2^{\left\lfloor r/2\right\rfloor+1}-\sum_{k=1}^{\left\lfloor r/2\right\rfloor}\xi_n(k)\geq 2^{\left\lfloor r/2\right\rfloor+1}-\sum_{k=1}^{\left\lfloor r/2\right\rfloor}2^k\geq 2,\] and (assuming $r$ is even) \[\sum_{d\mid_or}\mu(r/d)2^{\frac d2+1}=2^{\frac r2+1}+\sum_{\substack{d\mid_or \\ d<r}}\mu(r/d)2^{\frac d2+1}\geq 2^{\frac r2+1}-\sum_{\substack{d\mid_or \\ d<r}}2^{\frac d2+1}\geq 2^{\frac r2+1}-\sum_{k=2}^{r/2}2^k=4. \qedhere\] 
\end{proof}

Observe that the formulas for $\vert\text{Per}_r(T)\vert$ given in Theorem \ref{Thm3} do not depend on $n$ except through some divisibility restrictions. Consequently, if we fix $r$ and vary $T$, then $\vert\text{Per}_r(T)\vert$ can only attain finitely many values. More precisely, we have the following corollary. 

\begin{corollary}\label{Cor2}
Fix a positive integer $r$. Let \[\kappa_r=\sum_{d\mid r}\mu(r/d)2^d,\] \[\lambda_r=\begin{cases} \sum\limits_{d\mid r}\mu(r/d)\xi_{\frac r2+1}(d), & \mbox{if } r\equiv0\pmod2; \\ 2\kappa_r, & \mbox{if } r\equiv1\pmod2, \end{cases}\] and \[\rho_r=\sum\limits_{d\mid_or}\mu(r/d)2^{\frac d2+1}.\] Let $T=[C_n,\{f_{v_i}\}_{i=1}^n,\text{\emph{id}}]$, where $f_{v_i}\in\{\text{\emph{parity}}_3,(1+\text{\emph{parity}})_3\}$ for all $i$. If $T$ has a fixed point, then $\vert\text{\emph{Per}}_r(T)\vert\in\{0,\kappa_r,\lambda_r\}$. If $T$ does not have a fixed point, then $\vert\text{\emph{Per}}_r(T)\vert\in\{0,\rho_r\}$.  
\end{corollary} 
\begin{proof}
If $T$ does not have a fixed point, the result follows immediately from Theorem \ref{Thm3}. Thus, assume $T$ has a fixed point. Suppose further that $\vert\text{Per}_r(T)\vert\neq 0$. We will show that $\vert\text{Per}_r(T)\vert\in\{\kappa_r,\lambda_r\}$. 

First, suppose $n$ is even. Because $\vert\text{Per}_r(T)\vert\neq 0$, it follows from Theorem \ref{Thm3} that $r\mid n-1$ and $\displaystyle{\vert\text{Per}_r(T)\vert=\sum_{d\mid r}\mu(r/d)2^{d+1}}$. Because $r\mid n-1$ and $n$ is even, $r$ must be odd. Therefore, \[\vert\text{Per}_r(T)\vert=\sum_{d\mid r}\mu(r/d)2^{d+1}=2\kappa_r=\lambda_r.\]

Next, assume $n$ is odd. Because $\vert\text{Per}_r(T)\vert\neq 0$, we must have $r\mid 2n-2$ and \[\vert\text{Per}_r(T)\vert=\sum_{d\mid r}\mu(r/d)\xi_n(d)\] by Theorem \ref{Thm3}. We have two cases to consider. 

\noindent 
Case 1: In this case, assume $r\mid n-1$. For each divisor $d$ of $r$, we know that $\xi_n(d)=2^d$ because $d\mid n-1$. Hence, \[\vert\text{Per}_r(T)\vert=\sum_{d\mid r}\mu(r/d)2^d=\kappa_r.\] 

\noindent 
Case 2: In this case, assume $r\nmid n-1$. Note that this implies that $r$ is even. Since $r\mid 2n-2$, we deduce that $(r/2)\mid n-1$. Since $(r/2)\mid r$, this shows that $\gcd(r,n-1)\geq r/2$. Because $r\nmid n-1$, $\gcd(r,n-1)<r$. Therefore, $\gcd(r,n-1)=r/2$. Let $d$ be a positive divisor of $r$. We have $d\mid n-1$ if and only if $d\mid\gcd(r,n-1)$. That is, $d\mid n-1$ if and only if $d\mid(r/2)$. It follows that $\xi_n(d)=\xi_{\frac r2+1}(d)$. This holds for any positive divisor $d$ of $r$, so \[\vert\text{Per}_r(T)\vert=\sum_{d\mid r}\mu(r/d)\xi_n(d)=\sum_{d\mid r}\mu(r/d)\xi_{\frac r2+1}(d)=\lambda_r. \qedhere\]
\end{proof}

Theorem \ref{Thm3} gives an explicit description of the dynamics of $L_n=[C_n,\text{parity}_3,\text{id}]$ since the zero vector is certainly a fixed point of this map. Another natural SDS map worth considering is $[C_n,(1+\text{parity})_3,\text{id}]$, the SDS map obtained from setting all vertex functions equal to $(1+\text{parity})_3$. In the following theorem, we show that this map has a fixed point if and only if $n$ is a multiple of $4$. Hence, by Theorem \ref{Thm3}, we completely understand the dynamics of this map. The proof of the following theorem also shows that it is relatively simple to determine whether or not one of the symmetric invertible SDS maps that we have been considering has a fixed point. 

\begin{theorem}\label{Thm4}
The SDS map $[C_n,(1+\text{\emph{parity}})_3,\text{\emph{id}}]\colon\mathbb F_2^n\to\mathbb F_2^n$ has a fixed point if and only if $4$ divides $n$. 
\end{theorem}
\begin{proof}
Let $T_n=[C_n,(1+\text{parity})_3,\text{id}]$. Let \[\gamma_i=\begin{cases} 0, & \mbox{if } i\equiv 0\pmod 2; \\ 1, & \mbox{if } i\equiv 1\pmod 2. \end{cases}\] Put $u_n=(\gamma_1,\gamma_2,\ldots,\gamma_{n-1},\gamma_{n-1})$. For example, $u_8=(1,0,1,0,1,0,1,1)$. Using Lemma \ref{Lem1}, it is straightforward to check that $T_n(v)=L_n(v)+u_n$ for each $v\in\mathbb F_2^n$. 

Let $a_n$ be the $n$-tuple of $0$'s and $1$'s whose $i^\text{th}$ entry is $0$ if and only if $i\equiv 1,2\pmod 4$. For example, $a_8=(0,0,1,1,0,0,1,1)$. We leave the reader to verify that $a_n$ is a fixed point of $T_n$ if $4\mid n$. 

To prove the converse, suppose $T_n$ has a fixed point $c=(c_1,\ldots,c_n)$. We have $c=T_n(c)=L_n(c)+u_n$, so $u_n=L_n(c)+c$. For any $z\in\mathbb F_2^n$, let $[z]_j$ denote the $j^\text{th}$ entry of $z$. If $i\in\{1,\ldots,n-2\}$, then $[L_n(c)]_i=c_n+c_1+c_{i+1}$ by Lemma \ref{Lem1}. Therefore, $[L_n(c)+c]_i=c_n+c_1+c_i+c_{i+1}$. On the other hand, $[u_n]_i=\gamma_i$ by definition. Thus, if $1\leq i\leq n-2$, then
\begin{equation}\label{Eq5} 
c_n+c_1+c_i+c_{i+1}=\gamma_i.
\end{equation} Setting $i=1$ in \eqref{Eq5} yields $c_2+c_n=1$. Furthermore, \[\gamma_{n-1}=[u_n]_n=[L_n(c)+c]_n=c_2+c_n\] by Lemma \ref{Lem1}. It follows that $\gamma_{n-1}=1$, so $n$ is even. This implies that \[\sum_{i=1}^{n-2}(c_n+c_1+c_i+c_{i+1})=c_1+c_{n-1}=[L_n(c)+c]_{n-1}=[u_n]_{n-1}=\gamma_{n-1}=1.\] Invoking \eqref{Eq5}, we obtain \[1=\sum_{i=1}^{n-2}(c_n+c_1+c_i+c_{i+1})=\sum_{i=1}^{n-2}\gamma_i.\] This implies that $4\mid n$. 
\end{proof}

As in the proof of the previous theorem, let $T_n=[C_n,(1+\text{parity})_3,\text{id}]$. From Theorems \ref{Thm3} and \ref{Thm4}, we deduce that for any fixed positive integer $r$, there are only two possible values of $\vert\text{Per}_r(T_n)\vert$. More precisely, let \[\theta_r=\begin{cases} \sum\limits_{d\mid_or}\mu(r/d)2^{\frac d2+1}, & \mbox{if } r\equiv 0\pmod 2; \\ \sum\limits_{d\mid r}\mu(r/d)2^{d+1}, & \mbox{if } r\equiv 1\pmod 2. \end{cases}\] For any positive integer $n$, we have $\vert\text{Per}_r(T_n)\vert\in\{0,\theta_r\}$. 

Indeed, suppose $\vert\text{Per}_r(T_n)\vert\neq 0$. If $r$ is odd, then $r\nmid_o 2n-2$, so it follows from Theorem \ref{Thm3} that $T_n$ has a fixed point. By Theorem \ref{Thm4}, $n$ is a multiple of $4$. Using Theorem \ref{Thm3} once again, we find that $\vert\text{Per}_r(T_n)\vert=\sum\limits_{d\mid r}\mu(r/d)2^{d+1}=\theta_r$. Next, assume $r$ is even. If $4\mid n$, then Theorems \ref{Thm3} and \ref{Thm4} imply that $r\mid n-1$. This is impossible, so $4\nmid n$. Theorem \ref{Thm4} now tells us that $T_n$ does not have a fixed point, so $\vert\text{Per}_r(T_n)\vert=\sum\limits_{d\mid_or}\mu(r/d)2^{\frac d2+1}=\theta_r$ by Theorem \ref{Thm3}.    

\section{Suggestions for Future Research} 
Because SDS are defined using so many parts, there are several ways to vary the problems considered here. It is natural to consider symmetric invertible SDS defined using a different set of states (such as $\mathbb F_3$) or a different family of base graphs (such as wheel graphs or more general circulant graphs). In fact, it would be interesting to develop some method for gaining information about the invariant factor decompositions of SDS maps of the form $[G,\text{parity}_3,\pi]$, where $G$ is an element of some reasonably large family of base graphs and the implicit set of states is still $\mathbb F_2$. For example, finding the minimal polynomials of such maps would yield useful divisibility relationships governing the sizes of orbits of the maps. In doing so, it might be useful to consider the article \cite{Chen}, which proves results that allow one to compute the matrix of a linear SDS map when an explicit formula such as that derived in Lemma \ref{Lem1} is not readily available.

It is also worth mentioning the possibility of obtaining results similar to Corollary \ref{Cor2} and the final two paragraphs of Section 4. Consider the following definition. 
\begin{definition} \label{Def10} 
Let $I$ be an indexing set, and let $V=\{F_i\}_{i\in I}$ be a collection of self-maps of a set $X$. Let \[\mathcal R_r(V)=\{\vert\text{Per}_r(F_i)\vert\colon i\in I\}.\]
We say the collection $V$ is \emph{$r$-orbit-restricted} if $\mathcal R_r(V)$ is finite, and we say $V$ is \emph{orbit-restricted} if $V$ is $r$-orbit-restricted for all positive integers $r$. 
\end{definition}
Of course, any finite collection of functions from a set $X$ to itself is orbit-restricted. Let $\mathcal V$ denote the set of all SDS maps of SIAECA with the identity update order. According to Corollary \ref{Cor2}, $\mathcal R_r(\mathcal V)\subseteq\{0,\kappa_r,\lambda_r,\rho_r\}$ for each $r$. The two paragraphs at the end of Section 4 show that $\mathcal R_r(\{[C_n,(1+\text{parity})_3,\text{id}]\}_{n\geq 3})=\{0,\theta_r\}$. Hence, these collections of SDS maps are also orbit restricted. 

What other families of SDS maps are orbit-restricted? It seems quite difficult to answer this question in general, but we will at least make some conjectures. Fix an integer $m\geq 2$, and define $f_m\colon(\mathbb Z/m\mathbb Z)^3\rightarrow \mathbb Z/m\mathbb Z$ by $f_m(x_1,x_2,x_3)=x_1+x_2+x_3$. Furthermore, define $(1+f)_m\colon(\mathbb Z/m\mathbb Z)^3\rightarrow \mathbb Z/m\mathbb Z$ by $(1+f)_m(x_1,x_2,x_3)=1+x_1+x_2+x_3$. For the moment, we will define all SDS using the set of states $\mathbb Z/m\mathbb Z$ (instead of $\mathbb F_2$). Let $F_{m,n}$ and $G_{m,n}$ denote the SDS maps $[C_n,f_m,\text{id}]$ and $[C_n,(1+f)_m,\text{id}]$, respectively. 
\begin{conjecture} \label{Conj1} 
The collections $\{F_{m,n}\}_{n\geq 3}$ and $\{G_{m,n}\}_{n\geq 3}$ are each orbit-restricted. 
\end{conjecture}
It is reasonable to expect Conjecture \ref{Conj1} to hold when $m$ is a prime (so $\mathbb Z/m\mathbb Z$ is a field) since, in that case, linear algebra similar to that exploited above in the case $m=2$ could force a certain rigidity upon the dynamics of the maps $F_{m,n}$ and $G_{m,n}$. The conjecture is a bit more radical in the case in which $m$ is not a prime.  

\begin{conjecture} \label{Conj2}
Let $m,n,r$ be positive integers with $m\geq 2$ and $n\geq 3$. Every prime divisor of $\vert\text{\emph{Per}}_1(F_{m,n}^r)\vert$ divides $m$. Every prime divisor of $\vert\text{\emph{Per}}_1(G_{m,n}^r)\vert$ divides $m$.
\end{conjecture}
A small amount of numerical evidence\footnote{When $m=2$, Conjecture \ref{Conj3} follows from what we have already proven. We have also checked that Conjecture \ref{Conj3} holds for $n\leq 9$ when $m=4$ and for $n\leq 6$ when $m=8$.} hints at the following conjecture. 
\begin{conjecture} \label{Conj3}
Suppose $m$ is a power of $2$. If $F_{m,n}$ has a periodic orbit of period $r$, then $r\mid m(n-1)$. If $G_{m,n}$ has a periodic orbit of period $r$, then $r\mid m(n-1)$.
\end{conjecture} 
 
\section{Acknowledgments} 
The author would like to express his deepest gratitude to Taom Sakal, who wrote a computer program that was used to find the nonzero values of $\vert\text{Per}_r(T)\vert$ when $n\leq 17$ and $T=[C_n,\text{parity}_3,\text{id}]$ or $T=[C_n,(1+\text{parity})_3,\text{id}]$. This data was paramount to the production of conjectures that the author eventually developed into Theorem \ref{Thm3}. Mr. Sakal also used the program to produce data that was used to make Conjectures \ref{Conj2} and \ref{Conj3}.

The author thanks Dr. Padraic Bartlett for introducing him to sequential dynamical systems and giving excellent suggestions for the improvement of this article. 

The author wishes to thank the anonymous \emph{Discrete Mathematics} reviewers for very helpful suggestions. He would especially like to thank the reviewer who pointed out several relevant articles in the literature concerning connections with generalized toggle groups and techniques that helped streamline many of the proofs in this paper. 
 
\end{document}